\documentclass[11pt,a4paper]{article}
\usepackage{amsthm,amsmath, amssymb, amsfonts}
\usepackage[pdftex]{hyperref}
\usepackage[all]{xy}
\usepackage{graphicx}
\usepackage{epsfig}
\newtheorem{theorem}{Theorem}
\newtheorem{lemma}{Lemma}

\theoremstyle{remark}

\newtheorem*{ack}{Acknowledgements}

\newcommand*{\R}{\mathbb{R}}

\newcommand*{\N}{\mathbb{N}}

\newcommand*{\al}{\alpha}

\newcommand*{\rk}{\mathrm{rk}}

\newcommand*{\stab}{\mathrm{Stab}}
\newcommand*{\aut}{\mathrm{Aut}}
\newcommand*{\F}{\mathcal{F}}
\newcommand*{\G}{\mathcal{G}}

\title{The rank of edge connection matrices and the dimension of algebras of invariant tensors}
\author{Guus Regts\footnote{CWI, Amsterdam. Email: regts@cwi.nl.}}

\begin{document}

\maketitle
\begin{abstract}
We characterize the rank of edge connection matrices of partition functions of real vertex models, as the dimension of the homogeneous components of the algebra of $G$-invariant tensors. Here $G$ is the subgroup of the real orthogonal group that stabilizes the vertex model.
This answers a question of Bal\'azs Szegedy from 2007.
\\[.5cm]
\textbf{Keywords:} Edge connection matrix; partition function; vertex model; tensor algebra; orthogonal group; invariants.
\end{abstract}

\section{Introduction}
Let $\G$ be the set of all isomorphism classes of graphs, allowing multiple edges, loops and circles. 
Here a circle is a loop which does not contain vertices.
An $\R$-valued \emph{graph parameter} is a map $p:\G \to \R$ which takes the same values on isomorphic graphs.

Throughout this paper we set $\N=\{0,1,2\ldots\}$ and for $n\in \N$, $[n]$ denotes the set $\{1,\ldots,n\}$.
Let 
\begin{equation}
R:=\R[x_1,\ldots,x_n], 
\end{equation}
denote the polynomial ring in $n$ variables. 
Note that there is a one-to-one correspondence between linear functions $h:R\to \R$ and maps $h:\N^n\to \R$; $\alpha\in \N^n$ corresponds to the monomial $x^{\alpha}:=x_1^{\alpha_1}\cdots x_n^{\alpha_n}\in R$ and the monomials form a basis for $R$.
Let $D$ and  $E$ be finite sets, with $D\subseteq E$.
Given a map $\phi:E\to[n]$, we can view the multiset $\phi(D)$ as an element $\alpha \in \N^n$ via $\alpha_1=|\phi^{-1}\{1\}\cap D|,\ldots, \alpha_n=|\phi^{-1}\{n\}\cap D|$.

Following de la Harpe and Jones \cite{HJ} we call any map $h:\N^n\to \R$ a \emph{vertex model}. 
The \emph{partition function} of $h$ is the graph parameter $p_h:\G\to \R$ defined by
\begin{equation}
p_h(G)= \sum_{\phi: EG\to [n]} \prod_{v\in VG} h(\phi(\delta(v))),
\end{equation}
for $G\in \G$. Here $\delta(v)$ is the multiset of edges incident with $v$. Note that a loop counts twice.

We need to introduce the concept of a fragment. 
A \emph{$k$-fragment} is a graph which has $k$ of its vertices labeled $1$ up to $k$, each having degree one.
These labeled vertices are called the \emph{open ends} of the graph. 
The edge connected to an open end is called a \emph{half edge}. 
Let $\F_k$ be the collection of all $k$-fragments.
We can identify $\F_0$ with $\G,$ the collection of all graphs.
Let $k\in \N$.
Define a glueing operation $*:\F_k\times \F_{k}\to \G$ as follows: for $F,H\in \F_k$ connect the
half edges incident with open ends with identical labels to form single edges
(with the open ends erased); the resulting graph is denoted by $F*G$; see Figure~\ref{fig:gluing}.
Note that by glueing two half edges, of which both their endpoints are open ends, one creates a circle.

\begin{figure}
\begin{center}
\includegraphics[width=.6\textwidth]{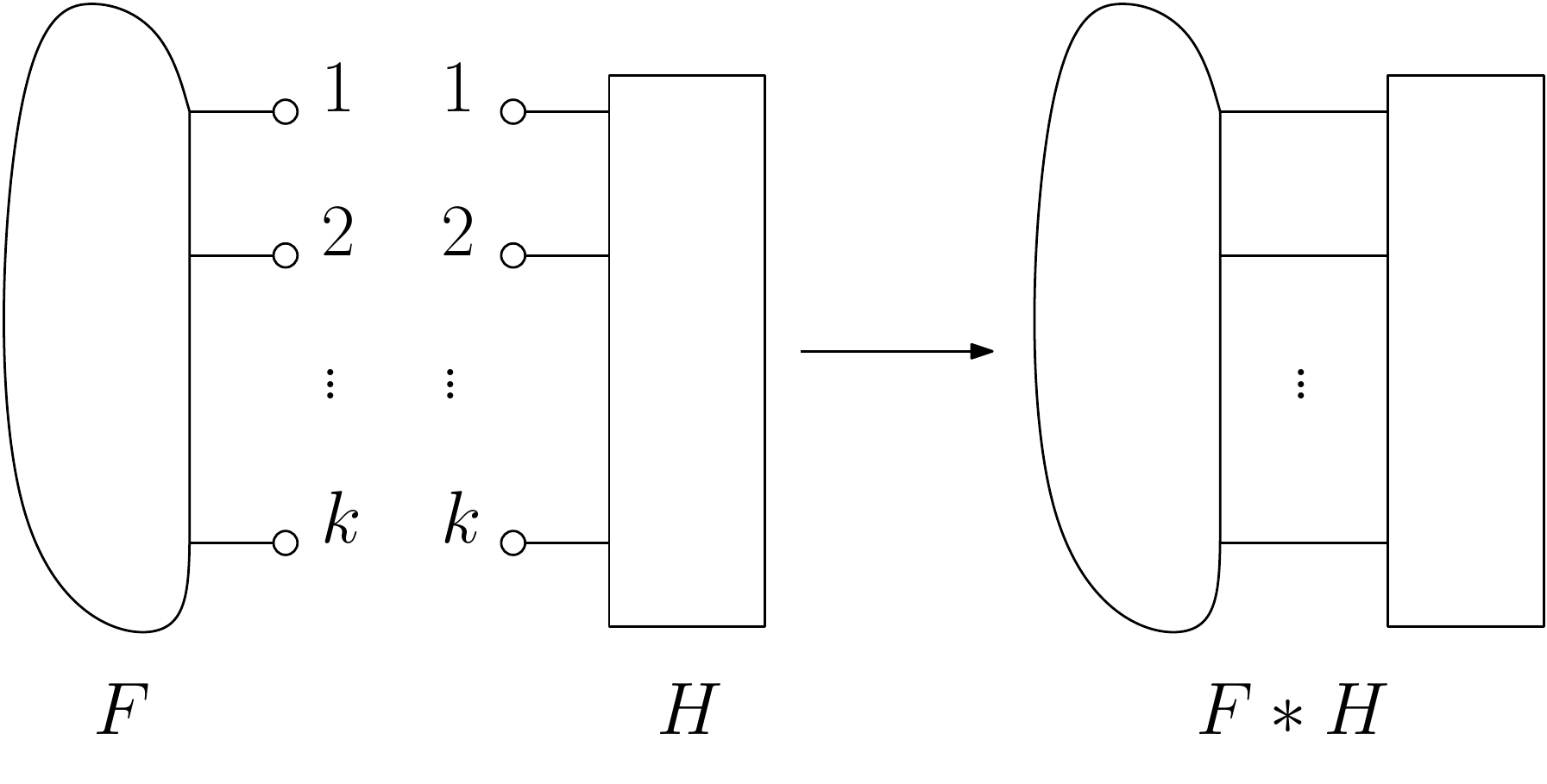}
\caption{Glueing two $k$-fragments into a graph.}
\label{fig:gluing}
\end{center}
\end{figure}

For any graph parameter $p,$ let $M_{p,k}$ be the $\F_k\times \F_k$ matrix defined by 
\begin{equation}
M_{p,k}(F,H)=p(F*H),
\end{equation}
for $F,H\in \F_k$.
This matrix is called the \emph{$k$-th edge connection matrix} of $p$.
In \cite{Szegedy}, Szegedy used the edge connection matrices to characterize when a graph parameter $p$ is of the form $p_h$ for some $h:\N^n\to \R$.
In the same paper (cf. \cite{Szegedy} Question 3.1), Szegedy asked for a characterization of $\rk(M_{p,k})$, for $k\in \N$, if $p=p_h$ for some $h:\N^n\to \R$.  
The same question is also posed in \cite{count}.
In this paper we give an answer to this question, which we describe now.

Let $O(n):=\{U\in \R^{n\times n}\mid U^TU=I\}$ be the orthogonal group.
The orthogonal group acts on the polynomial ring $R$; 
for $q\in R$, $v\in \R^n$ and $U\in O(n),$ $(Uq)(v):=q(U^{-1}v)$. 
The action of $O(n)$ on $R$ induces an action on $R^*$, the space of linear functions on $R$;
for $h\in R^*$, and $U\in O(n)$,  $Uh$ is defined by $(Uh)(q):=h(U^{-1}q)$, for $q\in R$.

For $h\in R^*$, define 
\begin{equation}
\stab(h):=\{U\in O(n)\mid Uh=h\}.
\end{equation}
Let $V=\R^n$.
The action of $O(n)$ on $V$ extends naturally to $V^{\otimes k}$ for any $k\in \N$.
As usual, for a subgroup $H$ of $O(n)$, 
\begin{equation}
(V^{\otimes k})^H:=\{v\in V^{\otimes k} \mid Uv=v \text{ for all } U\in H\}.
\end{equation}
Our characterization reads:

\begin{theorem}	\label{thm:main}
Let $h:\N^n\to \R$.
Then for any $k\in \N$,
\begin{equation}
\rk(M_{p_h,k})=\dim (V^{\otimes k})^{\stab(h)}.
\end{equation}
\end{theorem}

The organization of this paper is as follows. 
In Section 2 we discuss a related result by Lov\'asz, which characterizes the rank of vertex connection matrices of  partition functions of real valued weighted spin models.

Section 3 deals with some framework and preliminaries. In particular, we give a combinatorial interpretation of the space of tensors invariant under the action of $\stab(h)$, for $h:\N^n\to \R$ (cf. Theorem \ref{thm:main2}), from which we will deduce Theorem \ref{thm:main}.

In Section 4 we give a proof of Theorem \ref{thm:main2}. 
Our proof uses a result of Schrijver, characterizing algebras of $G$-invariant tensors for subgroups $G$ of $O(n)$.

\section{Related work}
In \cite{Lovasz}, Lov\'asz characterized the rank of vertex connection matrices of partition functions of real valued weighted spin models.
In order to state his result, we first need to introduce some terminology.
This is not used anywhere else in this paper. 

Let $\al\in \R^n$ be strictly positive and let $\beta\in \R^{n\times n}$ be symmetric.  
Following de la Harpe and Jones \cite{HJ} we call the pair $(\alpha,\beta)$ a \emph{weighted spin model}.
Let $\G'$ be the collection of all graphs, allowing multiple edges but no loops or circles. 
The \emph{partition function} of $(\alpha,\beta)$ is the graph parameter $p_{\alpha,\beta}:\G'\to \R$ defined by
\begin{equation}
p_{\al,\beta}(H):=\sum_{\phi:VH\to [n]}\prod_{v\in VH}\al_{\phi(v)}\cdot \prod_{uv\in EH} \beta_{\phi(u),\phi(v)},
\end{equation}
for $H\in \G'$.

We can view $p_{\alpha,\beta}$ in terms of weighted homomorphisms.
Let $G(\al,\beta)$ be the complete graph on $n$ vertices (including loops) with vertex weights given by $\al$ and edge weights given by $\beta$. 
Then $p_{\alpha,\beta}(H)$ can be viewed as counting the number of weighted homomorphisms of $H$ into $G(\al,\beta)$.
In this context $p_{\al,\beta}$ is often denoted by $\hom(\cdot,G(\alpha,\beta))$.

A \emph{$k$-labeled graph} is a graph $H$ with an injective map $\phi:[k]\to VH$.
Let $\G'_k$ be the collection of all $k$-labeled graphs.
For $F,H\in \G'_k$ define $FH\in \G'_k$ to be the $k$-labeled graph obtained  from the disjoint union of $F$ and $H$ by identifying equally labeled vertices.
For any graph parameter $p:\G'\to \R$ and $k\in \N$, the \emph{$k$-th vertex connection matrix} $N_{p,k}$ is the $\G'_k\times \G'_k$ matrix defined as follows:
\begin{equation}
N_{p,k}(F,H)=p(FH),
\end{equation}
for $F,H\in \G'_k$.

The vertex connection matrices were used by Freedman, Lov\'asz and Schrijver in \cite{FLS} to characterize graph parameters which are of the form $p_{\alpha,\beta}$ for some positive $\alpha\in \R^n$ and symmetric $\beta\in \R^{n\times n}$.

Let $\alpha\in \R^n$ positive and $\beta\in \R^n$ symmetric.
Lov\'asz \cite{Lovasz} characterized the ranks of the vertex connection matrices of the parameter $p_{\alpha,\beta}$.
Let $G=G(\alpha,\beta)$.
Two distinct vertices $i,j\in VG$ are called \emph{twins} if $\beta(i,l)=\beta(j,l)$ for all $l\in VG$.

\begin{theorem}	[Lov\'asz \cite{Lovasz}] \label{thm:lo}
Let $G$ be as above and let $\aut(G)\subseteq S_n$ be the automorphism group of $G$.
If $G$ is twin free, then
\begin{equation}
\rk(N_{p_{\alpha,\beta},k})=\dim (V^{\otimes k})^{\aut(G)}.
\end{equation}
\end{theorem}

If $G$ has twins, one can make a twin free graph $G'$ (by identifying twins) such that $\hom(H,G)=\hom(H,G')$ for all $H\in \G'$. 

Our result on the rank of the edge connection matrices of partition functions of real vertex models is a natural analogue of Lov\'asz' theorem.
It is possible to give a proof of Theorem \ref{thm:lo} using our proof method for Theorem \ref{thm:main}. 
We will however not do this here.

In \cite{lex spin}, Schrijver introduced a different type of vertex connection matrix. It is also possible to characterize the rank of these connection matrices of partition functions of unweighted spin models (i.e. $\alpha$ is the all ones vector) using our method. The characterization is similar to Theorem \ref{thm:lo}.

\section{Framework and preliminaries}
In this section we will introduce the necessary framework and preliminaries in order to give a proof of Theorem \ref{thm:main}.

Let $\F:=\bigcup_{k=0}^{\infty} \F_k$ and let $\R \F$ denotes the linear space consisting of (finite) formal $\R$-linear combinations of fragments.
These are called \emph{quantum fragments}.

Let $V=\R^n$ be equipped with the standard inner product $\langle \cdot,\cdot,\rangle$ and let $e_1,\ldots, e_n$ be the standard basis for $V$.
Let
\begin{equation}
T(V)=\bigoplus_{k=0}^{\infty} V^{\otimes k}.
\end{equation}
The space $T(V)$ is the \emph{tensor algebra} over $V$ (with product the tensor product). 
The standard basis for $V^{\otimes k}$ is given by the set $\{e_\phi\mid \phi:[k]\to [n]\}$, where $e_\phi:=e_{\phi(i)}\otimes e_{\phi(2)}\otimes \ldots \otimes e_{\phi(k)}$. 
The inner product on $V$ naturally extends to $T(V)$.

We now fix a vertex model $h:\N^n\to \R$ and show that there is a natural map from $\R\F$ to the tensor algebra.

Let $D\subseteq E$ be finite sets and let $\phi:D\to [n]$ and $\psi:E\to [n]$.
We say that $\psi$ \emph{extends} $\phi$ if $\phi(d)=\psi(d)$ for all $d\in D$.
This is denoted by $\psi \unrhd \phi$.
For a $k$-fragment $F$ and a map $\phi:[k]\to [n]$ we will consider $\phi$ as a map from the half edges of $F$ to $[n]$.
Moreover, for a fragment $F$ we denote by $VF$ its vertices which are not open ends and by $EF$ its edges, including half edges.
Define for $\phi:[k]\to [n]$ and $F\in \F_k$ 
\begin{equation}
h_{\phi}(F):=\sum_{\substack{\psi:EF\to[n]\\ \psi \unrhd \phi}} \prod_{v\in VF} h(\psi(\delta(v))),
\end{equation}
and extend this linearly to $\R\F_k$.

We now extend $p_h$ to a map from $\R\F$ to the tensor algebra.
Define  $p_h:\R\F\to T(V)$ by
\begin{equation}
p_h(F):= \sum_{\phi:[k]\to [n]}h_{\phi}(F)e_\phi,
\end{equation}
for $F\in \F_k$ and $k\geq 0$, and extend this linearly to $\R\F$. 
Note that for $k=0$ this agrees with our original definition of $p_h$.

Note that for $F,H\in \F_k$,
\begin{equation}
p_h(F*H)=\sum_{\phi:[k]\to [n]} h_{\phi}(F)h_{\phi}(H)=\langle p_h(F),p_h(H)\rangle.    \label{eq:hom prod}
\end{equation}
This implies that $M_{p_h,k}$ is the Gram matrix of the tensors $p_h(F)$ for $F\in \F_k$.
Hence
\begin{equation}
\rk(M_{p_h,k})=\dim(p_h(\R\F_k)).			\label{eq:rankisdim}
\end{equation}

Recall that $\stab(h)$ is the subgroup of $O(n)$ that leaves $h$ invariant.
The following theorem states that the image of $p_h$ is equal to the algebra of tensors invariant under $\stab(h)$;
thus giving a combinatorial interpretation of the algebra of tensors invariant under $\stab(h)$.

\begin{theorem}	\label{thm:main2}
Let $h:\N^n\to \R$.
Then
\begin{equation}
p_{h}(\R\F)=T(V)^{\stab(h)}.
\end{equation}
\end{theorem}

Note that Theorem \ref{thm:main} follows from Theorem \ref{thm:main2}, using \eqref{eq:rankisdim}.
So we only need to prove Theorem \ref{thm:main2}. 
This will be the content of the next section.

\section{A proof of Theorem \ref{thm:main2}}
A crucial ingredient in our proof is a result of Schrijver, characterizing algebras of $G$-invariant tensors for subgroups $G$ of $O(n)$.
To state this result we first need some definitions.

A subset $A$ of $T(V)$ is called \emph{graded} if $A=\bigoplus_{k=0}^\infty A\cap V^{\otimes k}$.

For $1\leq i< j\leq k\in \N$, the \emph{contraction}, $C^k_{i,j},$ is the unique linear map
\begin{align}
&C^k_{i,j}:V^{\otimes k}\to V^{\otimes k-2}, \text{   satisfying}
\\ 
v_1\otimes \ldots \otimes v_k&\mapsto \langle v_i,v_j\rangle v_1\otimes \cdots \otimes v_{i-1}\otimes v_{i+1} \cdots \otimes v_{j-1}\otimes v_{j+1}\otimes \cdots \otimes v_k.	\nonumber
\end{align}
A graded subset $A$ of $T(V)$ is called \emph{contraction closed} if $C^k_{i,j}(a)\in A$ for all $a\in A\cap V^{\otimes k}$ and $1\leq i< j\leq k\in \N$.

Note that for any subgroup $G\subseteq O(n)$, $T(V)^G$ is clearly graded.
It is also contraction closed.
Indeed, write for $v=v_1\otimes \ldots \otimes v_k$, $C^k_{i,j}(v)=\langle v_i,v_j\rangle u$. 
Then for any $U\in O(n)$ we have 
\begin{equation}
U(C^k_{i,j}(v))=\langle v_i,v_j\rangle Uu=\langle Uv_i,Uv_j\rangle Uu=C^k_{i,j}(Uv).
\end{equation}
Hence if $v\in T(V)^G$, then $C^k_{i,j}(v)\in T(V)^G$.

Define for a subset $A$ of $T(V)$, the \emph{stabilizer} of $A$ by
\begin{equation}
\stab(A):=\{U\in O(n)\mid Ua=a\text{ for all } a\in A\}.
\end{equation}

\begin{theorem}[Schrijver \cite{Sch}]		\label{thm:schrijver}
Let $A\subseteq T(V)$. 
Then $A=T(V)^{\stab(A)}$ if and only if $A$ is a graded contraction closed subalgebra of $T(V)$ and contains $\sum_{i=1}^n e_i\otimes e_i$.
\end{theorem}

Let $h:\N^n\to \R$ be a vertex model.
Our proof now consists of two steps.
First we show that $p_h(\R\F)$ satisfies the conditions of Theorem \ref{thm:schrijver}. 
After that we show that $\stab(p_h(\R\F))=\stab(h)$. 
Combining these two steps completes the proof of Theorem \ref{thm:main2}.

We can make $\R\F$ into a (graded) algebra by defining, for $F\in \F_k$ and $H\in \F_l$, $F\otimes H$ to be the disjoint union of $F$ and $H$, where we add $k$ to the labels of the open ends of $H$ so that $F\otimes H$ is contained in $\F_{k+l}$.

We define a contraction operation for fragments.
For $1\leq i< j\leq k\in \N$, the \emph{contraction} $\Gamma^k_{i,j}:\F_k\to\F_{k-2}$ is defined as follows:
for $F\in \F_k$, $\Gamma^k_{i,j}(F)$ is the $(k-2)$-fragment obtained from $F$ by connecting the half edges incident with the 
open ends labeled $i$ and $j$ to a single edge (erasing $i$ and $j$) and then relabeling the remaining open ends such that the order 
is preserved.

The following lemma shows that $p_h$ is an homomorphism of algebras and that $p_h$ preserves contractions.

\begin{lemma}	\label{lem:map p}
The map $p_h$ is a homomorphism of algebras. 
Furthermore, for $1\leq i<j\leq k \in \N$ and $F\in \F_k$,  
\begin{equation}
p_h(\Gamma^k_{i,j}(F))=C^k_{i,j}(p_h(F)).
\end{equation}
\end{lemma}

\begin{proof}
Let $F\in \F_k$ and $H\in \F_l$. Then 
\begin{align}
&p_h(F\otimes H)=\sum_{\phi:[k+l]\to[n]} h_{\phi}(F\otimes H) e_\phi=		\nonumber
\\
&\sum_{\substack{\phi:[k]\to[n]\\\psi:[l]\to[n]}}h_{\phi}(F)h_{\psi}(H)e_\phi\otimes e_\psi=p_h(F)\otimes p_h(H).
\end{align}

As for contractions, let $1\leq i<j\leq k$ and let $F\in \F_k$.
Note that for $\phi:[k]\to[n]$, the contraction of $e_\phi$ is contained in $\{e_{\psi}\mid \psi:[k-2]\to[n]\}$ if $\phi(i)=\phi(j)$ and is zero otherwise. The following equalities now prove the lemma:
\begin{align}
C^k_{i,j}(p_h(F))&=\sum_{\phi:[k]\to[n]}h_{\phi}(F) C^{k}_{i,j}(e_\phi)=\sum_{\substack{\phi:[k]\to[n]\\ \phi(i)=\phi(j)}}h_{\phi}(F) C^{k}_{i,j}(e_\phi)	\nonumber
\\
&= \sum_{\phi:[k-2] \to [n]}h_{\phi}(\Gamma^k_{i,j}(F))e_\phi=p_h(\Gamma^k_{i,j}(F)).
\end{align}
\end{proof}

Now to see that $p_h(\R\F)$ satisfies the conditions of Theorem \ref{thm:schrijver}, note that $p_h(\R\F)$ is clearly graded and by Lemma \ref{lem:map p} it is a contraction closed subalgebra of $T(V)$.
Moreover, let $I\in \F_2$ be the half edge of which both its endpoints are open ends.
Then $p_h(I)=\sum_{i=1}^n e_i\otimes e_i.$
Hence $p_h(\R\F)=T(V)^{\stab(p_h(\R\F)}$, by Theorem \ref{thm:schrijver}.

To conclude the proof, we will finally show that $\stab(h)=\stab(p_h(\R\F))$. 

The \emph{basic $k$-fragment $F_k$} is the $k$-fragment that contains one vertex and $k$ open ends connected to this vertex, labeled $1$ up to $k$.
For a map $\phi:[k]\to[n]$, we define the monomial $x^\phi\in R$ by $x^{\phi}:=\prod_{i=1}^k x_{\phi(i)}$. 
It is not difficult to see that
\begin{equation}
h_{\phi}(F_k)=h(x^{\phi}).
\end{equation}

\begin{lemma}\label{lem:basic}
Let $U\in O(n)$.
Then for any $\psi:[k]\to[n]$,
\begin{equation}
\langle U(p_h(F_k)),e_\psi\rangle=(Uh)(x^\psi).
\end{equation}
\end{lemma}

\begin{proof}
Write 
\begin{equation}
U^{-1}e_{\psi}=\sum_{\phi:[k]\to[n]}u_{\phi}e_{\phi}.
\end{equation}
Then
\begin{equation}
U^{-1}x^{\psi}=\sum_{\phi:[k]\to[n]}u_{\phi}x^{\phi}.
\end{equation}
We then have
\begin{align}
\langle U(p_h(F_k)),e_\psi\rangle&=\langle p_h(F_k),U^{-1}e_\psi \rangle=\sum_{\phi:[k]\to [n]}h(x^{\phi})\langle e_\phi,U^{-1}e_\psi \rangle	\nonumber
\\
&=\sum_{\phi:[k]\to[n]}h(x^{\phi})u_{\phi}=h(U^{-1}x^{\psi})=(Uh)(x^{\psi}).	\label{eq:subgroup1}
\end{align}
This proves the lemma.
 
\end{proof}
Lemma \ref{lem:basic} immediately implies that $\stab(p_h(\R\F))\subseteq \stab(h)$.
Moreover, it also implies that $p_h(F_k)\in T(V)^{\stab(h)}$ for any $k$.
Note that $p_h(I)=\sum_{i=1}^n e_i\otimes e_i\in T(V)^{\stab(h)}$, as $\stab(h)\subseteq O(n)$.
Since any fragment $F\in \F$ can be obtained as a series of contractions applied to products of basic fragments and $I$
(see Figure\ref{fig:contract}), it follows that $p_h(\R\F)\subseteq T(V)^{\stab(h)}$.
Hence $\stab(h)\subseteq \stab(p_h(\R\F))$.

\begin{figure}	
\begin{center}
\includegraphics[width=.75\textwidth]{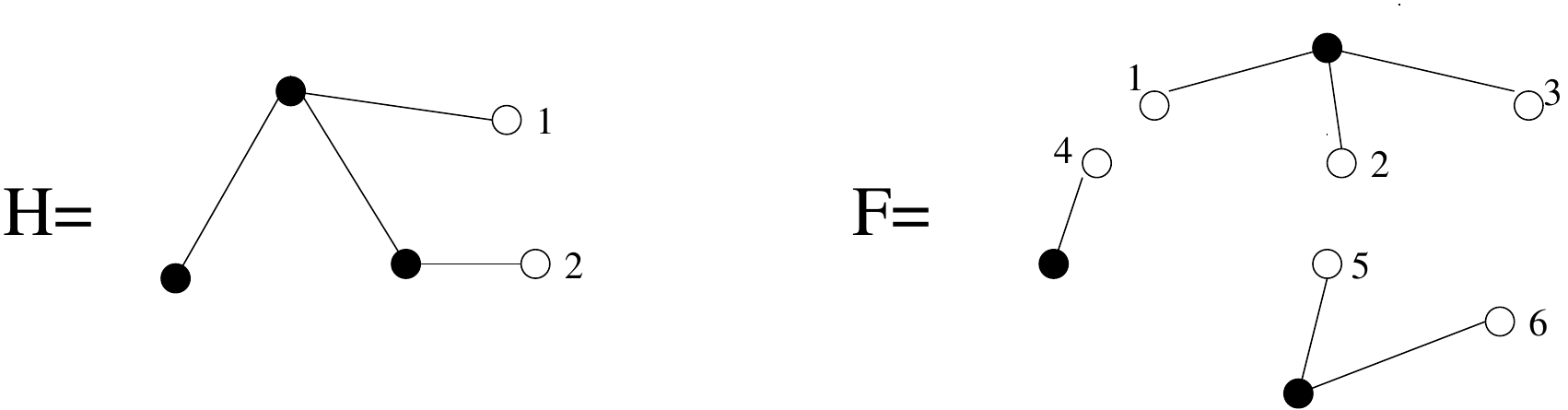}
\caption{The fragment $H$ can be obtained from the fragment $F$ by applying first $\Gamma^{6}_{1,4}$ and then $\Gamma^4_{1,3}$.}
\label{fig:contract}
\end{center}
\end{figure}

 \begin{ack}
I thank Lex Schrijver and Dion Gijswijt for useful discussions and helpful remarks as to the presentation of the results.
\end{ack}

\end{document}